\documentclass[12pt,reqno,centertags,draft]{amsart}
\usepackage{amsmath}
\usepackage{}
\usepackage{color}
\usepackage{amsfonts}
\usepackage{a4wide}
\usepackage{amsmath,amsthm,amssymb,amscd}
\usepackage{latexsym}
\usepackage{hyperref}
\usepackage[numbers,sort&compress]{natbib}
\usepackage{hypernat}
\allowdisplaybreaks
\numberwithin{equation}{section}

\newtheorem{theorem}{Theorem}[section]
\newtheorem{proposition}[theorem]{Proposition}

\newtheorem{lemma}[theorem]{Lemma}

\theoremstyle{definition}

\newcommand{\va}{\varepsilon}

\newcommand{\ds}{\displaystyle}
\newcommand{\sik}{\sum_{i=1}^k}

\def\r{\mathbb{R}}
\begin{document}
\title
[ Chern-Simons-Schr\"{o}dinger systems]
{Multi-Peak solutions to Chern-Simons-Schr\"{o}dinger systems with non-radial potential}

\maketitle
\begin{center}
\author{Jin Deng,\ \  Wei Long,\ \ Jianfu Yang }
\footnote{ Email addresses: jindeng\_2016@126.com, lwhope@jxnu.edu.cn, jfyang\_2000@yahoo.com}
\end{center}
\begin{center}

\address{College of Mathematics
and Information Science, Jiangxi Normal University, Nanchang,
Jiangxi 330022, P. R. China }

\end{center}

\begin{abstract}
In this paper, we consider the existence of static solutions to the nonlinear Chern-Simons-Schr\"{o}dinger system
\begin{equation}\label{eqabstr}  \left\{\begin{array}{ll}
-ihD_0\Psi-h^2(D_1D_1+D_2D_2)\Psi+V\Psi=|\Psi|^{p-2}\Psi,\\
\partial_0A_1-\partial_1A_0=-\frac 12ih[\overline{\Psi}D_2\Psi-\Psi\overline{D_2\Psi}],\\
\partial_0A_2-\partial_2A_0=\frac 12ih[\overline{\Psi}D_1\Psi-\Psi\overline{D_1\Psi}],\\
\partial_1A_2-\partial_2A_1=-\frac12|\Psi|^2,\\
\end{array} \right.
\end{equation}
where $p>2$ and non-radial potential $V(x)$ satisfies some certain conditions.
We show  that for every positive integer $k$, there exists $h_0>0$ such that for $0<h<h_0$, problem \eqref{eqabstr} has a nontrivial static solution $(\Psi_h, A_0^h, A_1^h,A_2^h)$. Moreover, $\Psi_h$ is a positive non-radial function with $k$ positive peaks, which approach to the local maximum point of $V(x)$ as $h\to 0^+$.



\end{abstract}

\section{Introduction}
In this paper, we investigate the existence of mutil-peak solutions to Chern-Simons-Schr\"{o}dinger systems. The Schr\"{o}dinger equation
\begin{equation}\label{eq0a}
ih\frac{\partial\Psi(x,t)}{\partial t}=-h^2\Delta\Psi(x,t)+V(x)\Psi(x,t)-|\Psi(x,t)|^{p-2}\Psi(x,t)
\end{equation}
with $p>2$ in $\mathbb{R}^{2}\times\mathbb{R}_+$ can be introduced as the Euler-Lagrange equation of the Lagrange density
\begin{equation}\label{eq0b}
\mathcal{L}=hRe\{i\bar{\Psi}(x,t)\frac{\partial\Psi(x,t)}{\partial t}\}-h^2|\nabla\Psi(x,t)|^2+V(x)|\Psi(x,t)|^2-\frac 2p|\Psi(x,t)|^p,
\end{equation}
where $V(x)$ is the external potential, $h$ is the Plank constant. A static solution, that is a solution $\Psi(x,t)=u(x)$ of \eqref{eq0a} which is independent of $t$, satisfies
the semiclassical Schr\"{o}dinger equation
\begin{equation}\label{eq0c}
-\varepsilon^2\Delta u(x)+V(x)u(x)=|u(x)|^{p-2}u(x),\quad x\in \mathbb{R}^2
\end{equation}
with $\varepsilon=h$.

Taking into account the interaction of the electromagnetic field and the matter field, one includes the Chern-Simons term into the Lagrangian density. The Lagrangian density then becomes
\begin{equation}\label{eq0d}
\begin{split}
\mathcal{L}_c&=\frac{\kappa}4\varepsilon^{\mu\alpha\beta} A_\mu F_{\alpha\beta}-\frac 12hRe\{i\bar{\Psi}(x,t)D_0\Psi(x,t)\}+\frac{h^2}{2}|D\Psi(x,t)|^2\\
&+\frac 12V(x)|\Psi(x,t)|^2-\frac 1p|\Psi(x,t)|^p,\\
\end{split}
\end{equation}
where $\Psi: \r^{2,1}\rightarrow \mathbb{C}$ is the complex scalar field, $A_{\mu}: \r^{2,1}\rightarrow \r, \mu =0,1,2,$ are the gauge field, which obey the Lorentz condition $\sum_{\mu=0}^2\partial_\mu A_\mu=0$. By $D_{0}=\partial_{t}+\frac ih A_{0}$ and $D_{j}=\partial_{x_j}-\frac ihA_{j}, \, j =1,2,$ for $(x_1,x_2,t)\in \r^{2,1}$ we denote the the covariant derivatives,  and we set $F_{\mu\nu}=\partial_\mu A_\nu-\partial_\nu A_\mu$ for $\mu,\nu=0,1,2.$ Inside the Lagrangian density $\mathcal{L}_c$  we denote $i$ the the imaginary unit, and $-\frac{1}4\varepsilon^{\mu\alpha\beta} A_\mu F_{\alpha\beta}$ the Chern-Simons term. The corresponding Euler-Lagrange system of $\mathcal{L}_c$ is given as follows.
\begin{equation}\label{eq0e} \left\{\begin{array}{ll}
-ihD_0\Psi-h^2(D_1D_1+D_2D_2)\Psi+V\Psi=|\Psi|^{p-2}\Psi,\\
\partial_0A_1-\partial_1A_0=-\frac 12ih[\overline{\Psi}D_2\Psi-\Psi\overline{D_2\Psi}],\\
\partial_0A_2-\partial_2A_0=\frac 12ih[\overline{\Psi}D_1\Psi-\Psi\overline{D_1\Psi}],\\
\partial_1A_2-\partial_2A_1=-\frac12|\Psi|^2.\\
\end{array} \right.
\end{equation}
 The systems \eqref{eq0e} is proposed in \cite{d,jp,jp1,jp2}, which describes the dynamics
of large number of particles in an electromagnetic field. This model is important for the study of the high-temperature superconductor,
fractional quantum Hall effect and Aharovnov-Bohm scattering. System \eqref{eq0e} is referred to  be the Chern-Simons-Schr\"{o}dinger system (CSS),
it is invariant under the following gauge transformation
$$
\phi\rightarrow \phi e^{i\chi},\ A_{\mu}\rightarrow A_{\mu}-\partial_{\mu}\chi
$$
for arbitrary $C^{\infty}$ function $\chi: \r^{2,1}\rightarrow\r$.

Since system  \eqref{eq0e} is setting in the whole space, a problem of the loss of the compactness is then raised if the variational method applied. In order to avoid such a problem,  in \cite{bhs} a particular form of solutions of \eqref{eq0e}
$$
\Psi(t,x)=u(|x|)e^{i\omega t}, \ \ A_0(t,x)=h_1(|x|),
$$
$$
A_1(t,x)=\frac{x_2}{|x|^2}h_2(|x|),\ \ A_2(t,x)=\frac{x_1}{|x|^2}h_2(|x|),
$$
is considered  with the constant $V$,  where $\omega>0$ and $u,h_1,h_2$ are real value functions depending only on $|x|$. Then, solutions are found in the radially symmetric space $H_r^1(\mathbb{R}^2)$ as critical points of the associated functional
$$
J(u)=\frac12\int_{\mathbb{R}^2}\bigg\{|\nabla u|^2+(\omega+\xi)u^2+\frac{u^2}{|x|^2}\bigg(\int_0^{|x|}\frac{s}2u^2(s)\,ds\bigg)^2\bigg\}\,dx-\frac{1}{p}\int_{\mathbb{R}^2}|u|^p\,dx.
$$
However, it is quite involved in finding critical points of $J$. Actually, such a problem was treated differently in accordance to the range of the exponent $p$. Precisely in \cite{bhs}, for $p>4$ it is considered a minimization problem on the Nehari-Pohozaev manifold; while for $2<p<4$, minimization problem is constrained in $L^2$ sphere. Essentially,  it is  a nonlinear eigenvalue problem.
For the case $p=4$, a self-dual solution can be found by Liouville equations. Later on, for $p \in(2, 4)$  Pomponio and Ruiz studied in \cite{pr} the
nonexistence and multiplicity results for problem \eqref{eq0e} with constant potentials $V$ by investigating the geometry of the Euler-Lagrange
functional. More related results on problem \eqref{eq0e} with constant potentials can be found
in \cite{bhs,bhs1,cdpg,h,pr,pr1} and references therein. Suppose $V$ is radially symmetric, it is studied in \cite{wt0} the existence, nonexistence and multiplicity of the same type of solutions for \eqref{eq0e}. While for the case $V$ being non-radial, nontrivial solutions are found in \cite{wt} under the assumption $p>4$.

In this paper, we consider the existence and concentration of solutions for \eqref{eq0} under the assumption that $V(x)$ is non-radial and $p>2$. Since $V(x)$ is not radially symmetric, we encounter the difficulty of the loss of compactness. Instead of working in the radially symmetric space $H_r^1(\mathbb{R}^2)$, one has to use the concentration-compactness principle if minimization problem considered.  Furthermore, the presence of $V$ brings essential difficulty to find minimizers of associated functional on the Nehari-Pohozaev manifold, it is also the reason hinders the study of the case $p\in(2,4]$ although there are existence results emerged for the case $p>4$. Our approach is to use the Lyapunov-Schmidt reduction method. By this method, we can not only  construct  mutil-peak solutions to problem \eqref{eq0e}, but also treat all the case $p>2$ in a unified way.

In this paper, we study static solutions of \eqref{eq0e}. Hence, the gauge field $(A_1,A_2)$ obeys the Coulomb condition $\partial_1A_1+\partial_2A_2=0$. Moreover, a static solution $(u, A_0, A_1,A_2)$ satisfies
\begin{equation}\label{eq0} \left\{\begin{array}{ll}
-\varepsilon^2\Delta u +V(x)u+ A_0u+(A_1^2+A_2^2)u = |u|^{p-2}u,\\
\partial_1A_0=A_2u^2,\ \partial_2A_0=-A_1u^2\\
\partial_1A_2-\partial_2A_1=-\frac12|u|^2,\ \partial_1A_1+\partial_2A_2=0,
\end{array} \right.
\end{equation}
where we set $\varepsilon =h>0$.

Suppose the external potential $V(x)$ satisfies
\begin{itemize}
\item[($V_1$)]  $V(x) \in C(\r^2,\r)$, $\inf_{x\in\r^2}V(x)>0$, and there exist positive constants $L$ and $\theta$ such
that $|V(x)-V(y)|\leq L|x-y|^{\theta}$ for all $x,\ y\in
\r^2$;
\item[($V_2$)] There exist $\delta>0$ and $x^0\in \r^2$ such that
$V(x)<V(x^0)$ for $x\in B_{\delta}(x^0)\setminus
\{x^0\}\subset\r^2$.
\end{itemize}

\bigskip

The main result of this paper is  as follows.
\begin{theorem}\label{th1}
 Suppose that $p>2$ and $V(x)$ satisfies $(V_1)$ and $(V_2)$. Then for any positive integer $k$, there exists $\va_0>0$ such that if $0<\va<\va_0$, problem \eqref{eq0} has a nontrivial solution $(u_\varepsilon, A_0^\varepsilon, A_1^\varepsilon,A_2^\varepsilon)$ such that $u_{\varepsilon}$ is positive non-radial function with k positive peaks, which approach to the local maximum point of $V(x)$ as $\va\to 0^+$.
\end{theorem}

The existence of peak solutions and the concentration phenomenon of solutions has been extensively studied for the semiclassical Shr\"{o}dinger equation \eqref{eq0c}, see for instance \cite{ABC, AMN, PF, PF1, WZ} for part of results in this direction. In particular, positive solutions with prescribed number of peaks to nonlinear
Schr\"{o}dinger equations were obtained in \cite{kw} by using the well-known Lyapunov-Schmidt reduction scheme. This argument  was later
generalized in \cite{cp} to study  nonlinear Schr\"{o}dinger equations with vanishing potentials. However, it seems no multi-peaks solution has been found for problem \eqref{eq0} in literatures.

 We will find solutions of problem \eqref{eq0} by looking for critical points of the associated functional
\begin{equation} \label{2-function}
\begin{split}
J_{\va}(u,A_0,A_1,A_2)&=\frac12\int_{\r^2}\big(\va^2|\nabla u|^2+V(x)|u|^2+(A_0+A_1^2+A_2^2)|u|^2\big)\,dx\\
&+\frac 12\int_{\r^2}(A_0F_{12}+A_1\partial_2A_0-A_2\partial_1A_0)\,dx-\frac1p\int_{\r^2}|u|^p\,dx.\\
\end{split}
\end{equation}
Such a problem can be reduced, see section 2 for details,  to find critical points of the functional
\begin{equation} \label{2-function-1}
\begin{split}
I_{\va}(u) = &  \frac{1}{2}\int_{\mathbb{R}^{2}} \bigl(\varepsilon^{2} |\nabla u|^{2} + V(x) u^{2}\bigr) dx + \frac12 \int_{\mathbb{R}^{2}} \biggl(- \frac{1}{4\pi} \int_{\mathbb{R}^{2}} \frac{x_{2} - y_{2}}{|x- y|^{2}} u^{2}(y) dy\biggr)^{2} u^{2}(x) dx \\
& + \frac12 \int_{\mathbb{R}^{2}} \biggl( \frac{1}{4\pi} \int_{\mathbb{R}^{2}} \frac{x_{1} - y_{1}}{|x- y|^{2}} u^{2}(y) dy\biggr)^{2} u^{2}(x) dx -\frac{1}{p} \int_{\mathbb{R}^{2}} |u|^{p} dx.
\end{split}
\end{equation}
Using the unique ground state $U$  of
\begin{equation} \label{eq1} \left\{\begin{array}{l@{\quad }l}
 -\Delta u+V(x^0) u=u^{p-1},\ u>0,\ x \in \r^2,\vspace{2mm} \\
  u(0)=\max\limits_{\r^N}u(x),\quad u\in H^1(\r^2),
\end{array} \right.
\end{equation}
we build up the approximate critical points for the functional $I_\va$.
It is well-known that  $U(x)=U(|x|)$ is non-degenerate and  satisfies
\begin{equation} \label{eq1a}
 U'(r)<0, \,\, \lim_{r\rightarrow \infty}r^{\frac{N-1}{2}}e^rU(r)=C>0,\,\,
\lim_{r\rightarrow \infty}\frac{U'(r)}{U(r)}=-1.
\end{equation}

Let $k$ be any positive integer.  Define
\[
\begin{split}
D_{k}^{\varepsilon,\delta}=&\Bigl\{{\bf {y}}=(y^1,\cdots,y^k)\in (\r^2)^k: y^i\in B_{\frac{\delta}{2}}(x^0),\quad {\rm and}\\
 &\quad\frac{|y^i-y^j|}{\va}\geq|\ln\va|^{\frac12},\ i\neq j,\,i,j=1,2\cdots,k\Bigr\}.\\
\end{split}
\]
Let $U_{\va, y^i}(x) =U(\frac{x-y^i}{\va})$ and
\[
E:=\Bigl\{\varphi\in H^1(\r^2): \langle\frac{\partial U_{\va, y^i}}{\partial y_l^i} \varphi\rangle_{\va}=0,\ i=1,\cdots,k,\ l=1,2\Bigr\},\\
\]
where
\begin{eqnarray*}
\langle v_1,v_2\rangle_{\va}:=\va^2\int_{\r^2}\nabla v_1 \nabla v_2\,dxdy+\int_{\r^2}v_1(x)v_2(x)dx.
\end{eqnarray*}
Fixing ${\bf {y}} \in D_k^{\va,\delta}$, we set
$$
U_{*}=\sum\limits_{i=1}^kU_{\va,y^i}=\sum\limits_{i=1}^kU(\frac{x-y^i}{\va})
$$
and
$$
M_k^\va=\Bigl\{( {\bf {y}}, \varphi):\,\,{\bf {y}} \in D_k^{\va,\delta},\,\,\,\,\varphi\in E\Bigr\}.
$$
Then, Theorem \ref{th1} will be proved by the following result.
\begin{theorem}\label{th2}
If $V(x)$ satisfies $(V_1)$ and $(V_2)$, then for any positive integer $k$, there is a positive constant $\va_0$ depending on $k$ such that for each $\va\in (0,\va_0],$
the functional $I_\va$ has a positive critical point of the form
 $$
 u_\va(x)=U_*+\varphi_\va=\sum\limits_{i=1}^kU_{\va,y^i}+\varphi_\va,
 $$
 where $\varphi_\varepsilon\in E$.
\end{theorem}

To prove Theorem \ref{th2}, we first use Lyapunov-Schmidt reduction scheme to reduce the problem to a variational problem defined on a closed subset of a finite
dimensional Euclidian space. Then we prove that the functional achieves its maximum in the interior of
that closed subset. Comparing with previous works, see for instance \cite{cp,kw}, the Chern-Simons term brings new difficulties in employing
the reduction method to deal with singularly perturbed problem \eqref{eq0},  and the feature of Chern-Simons term requires to establish some new delicate estimates on the energy of the approximate solutions.

This paper is organized as follows. After some preparation in section 2, we expand the functional $I_\va$ at $U_*+\varphi_\va$  and fundamental estimates are established in section 3. Finally, in section 4, we prove Theorem\ref{th2} by the reduction method.

 \section{Preliminaries}
In this section, we present the variational framework and establish estimates for the energy functional $I_\va$ at $U_*$.

By equation \eqref{eq0}, $A_0,A_1$ and $A_2$ can be expressed as functions of $u$. First, integrating by part we find
\begin{equation} \label{2-3}
\begin{split}
J_{\va}(u,A_0,A_1,A_2)&=\frac12\int_{\r^2}\big(\va^2|\nabla u|^2+V(x)|u|^2+(A_0+A_1^2+A_2^2)|u|^2\big)\,dx\\
&+\int_{\r^2}A_0F_{12}\,dx-\frac1p\int_{\r^2}|u|^p\,dx.\\
\end{split}
\end{equation}
Equation \eqref{eq0} implies
\begin{equation} \label{2-3a}
\int_{\r^2}A_0F_{12}\,dx = -\frac 12\int_{\r^2}A_0|u|^2\,dx.
\end{equation}
Hence,
\begin{equation} \label{2-4}
J_{\va}(u,A_0,A_1,A_2)=\frac12\int_{\r^2}\big(\va^2|\nabla u|^2+V(x)|u|^2+(A_1^2+A_2^2)|u|^2\big)\,dx
-\frac1p\int_{\r^2}|u|^p\,dx.\\
\end{equation}

Next, the Coulomb condition $\partial_1A_1+\partial_2A_2=0$ and the equation $\partial_1A_2-\partial_2A_1 = -\frac 12|u|^2$ yield
\begin{equation} \label{2-4a}
\Delta A_1=\frac{1}{2}\partial_2(|u|^2),\ -\Delta A_2=\frac{1}{2}\partial_1(|u|^2).
\end{equation}
Solving equation \eqref{2-4a} we obtain
\begin{equation} \label{2-5}
A_1=A_1(u)=\frac{1}{2}K_2*(\partial_2|u|^2)=-\frac{1}{4\pi}\int_{\r^2}\frac{x_2-y_2}{|x-y|^2}|u(y)|^2\,dy,
\end{equation}
\begin{equation} \label{2-6}
A_2=A_2(u)=-\frac{1}{2}K_1*(\partial_1|u|^2)=\frac{1}{4\pi}\int_{\r^2}\frac{x_1-y_1}{|x-y|^2}|u(y)|^2\,dy,
\end{equation}
where $K_i=\frac{-x_i}{2\pi|x|^2},$ for $i=1,2$ and $*$ denotes the convolution. Similarly, the equation
\[
\Delta A_0=\partial_1(A_2|u|^2)-\partial_2(A_1|u|^2)
\]
implies that
\begin{equation} \label{2-7}
A_0=A_0(u)=K_1*(A_1|u|^2)-K_2*(A_2|u|^2).
\end{equation}
Therefore, the functional $J_\va$ can be written as
\begin{equation} \label{2-function-1}
\begin{split}
I_{\va}(u) = &  \frac{1}{2}\int_{\mathbb{R}^{2}} \bigl(\varepsilon^{2} |\nabla u|^{2} + V(x) u^{2}\bigr) dx + \frac12 \int_{\mathbb{R}^{2}} \bigl(- \frac{1}{4\pi} \int_{\mathbb{R}^{2}} \frac{x_{2} - y_{2}}{|x- y|^{2}} u^{2}(y) dy\bigr)^{2} u^{2}(x) dx \\
& + \frac12 \int_{\mathbb{R}^{2}} \bigl( \frac{1}{4\pi} \int_{\mathbb{R}^{2}} \frac{x_{1} - y_{1}}{|x- y|^{2}} u^{2}(y) dy\bigr)^{2} u^{2}(x) dx -\frac{1}{p} \int_{\mathbb{R}^{2}} |u|^{p} dx.
\end{split}
\end{equation}
We know that $I_{\va}$ is well defined in $H^1(\r^2)$ and $I_{\va}\in C^1(H^1(\r^2))$. If $u$ is a critical point  of $I_{\va}$, we may define $A_0,A_1,A_2$ through \eqref{2-5}, \eqref{2-6} and \eqref{2-7}, then $(u,A_0,A_1,A_2)$ is a solution of problem \eqref{eq0}. In the following, we focus on finding critical points
of the functional $I_{\va}$. Precisely, we are looking for critical point $u$ of $I_{\va}$ in the form
\[
u = U_*+\varphi
\]
with $\varphi\in E$. To this purpose, we expand the functional $I_\va$ near  approximate solutions.
\begin{proposition}\label{energy-prop} There holds,
\begin{eqnarray*}
I_{\va}(U_*)&= &(\frac12 -\frac1p)k\va^{2} \int_{\r^{2}} U^{p}dx -\frac12 \sik \big(V(x^0)-V(y^i)\big)\va^2\int_{\r^2}U^2 dx - C \va^2\sum_{i\neq j}^ke^{\frac{-|y^i-y^j|}{\va}} \\
& +& O\big(\va^{2+\theta} +\va^2\sum\limits_{i\neq j}^ke^{\frac{-(1+\sigma)|y^i-y^j|}{\va}} + \va^{4}  \big),
\end{eqnarray*}
where $C>0$ and $\sigma$ is a small fixed constant.
\end{proposition}
\begin{proof}
Since the integral $\int_{\mathbb{R}^{2}}\frac{1}{|x- y|} U^{2}(y)\,dy$ is uniformly bounded in $x\in \mathbb{R}^{2}$, a change of variable yields
$$
\int_{\mathbb{R}^{2}}\frac{1}{|x- y|} U^{2}(\frac{y-y^{i}}{\varepsilon}) dy =\va \int_{\mathbb{R}^{2}} \frac{1}{|y- \frac{x - y^{i}}{\va}|} U^{2}(y) dy \leq C \va.
$$
This implies
\begin{equation}\label{energy-prop1-1}\begin{split}
&| \int_{\mathbb{R}^{2}} \bigl(- \frac{1}{4\pi} \int_{\mathbb{R}^{2}} \frac{x_{2} - y_{2}}{|x- y|^{2}} U_{*}^{2}(y) dy\bigr)^{2} U_{*}^{2}(x) dx|\\
\leq & C \int_{\mathbb{R}^{2}} \bigl( \int_{\mathbb{R}^{2}} \frac{1}{|x- y|} (\sum_{i=1}^{k}U_{\varepsilon, y^{i}}(y) )^{2}\,dy\bigr)^{2} (\sum_{i=1}^{k}U_{\varepsilon, y^{i}}(x))^{2}\, dx\\
\leq & C \int_{\mathbb{R}^{2}} \bigl(\sum_{i=1}^{k} \int_{\mathbb{R}^{2}}\frac{1}{|x- y|} U_{\varepsilon, y^{i}} ^{2}(y)\, dy\bigr)^{2} (\sum_{i=1}^{k}U_{\varepsilon, y^{i}}(x) )^{2}\, dx\\
 \leq & C \va^{2}\int_{\mathbb{R}^{2}}(\sum_{i=1}^{k}U_{\varepsilon, y^{i}}(x) )^{2}\, dx\\
  \leq & C \va^{4}.
\end{split}
\end{equation}
Analogously,
\begin{equation}\label{energy-prop1-1'}\begin{split}
 \frac12 \int_{\mathbb{R}^{2}} \bigl( \frac{1}{4\pi} \int_{\mathbb{R}^{2}} \frac{x_{1} - y_{1}}{|x- y|^{2}} U_{*}^{2}(y) dy\bigr)^{2} U_{*}^{2}(x) dx \leq C \va^{4}.
\end{split}
\end{equation}
Since
\[\va^{2}\int_{\r^{2}} \nabla U_{\va, y^{i}} \nabla U_{\va, y^{j}}\,dx + V(x^{0}) \int_{\r^{2}} U_{\va, y^{i}} U_{\va, y^{j}}\,dx = \int_{\r^{2}} U^{p-1}_{\va, y^{i}} U_{\va, y^{j}}dx\]
for any $i, j = 1, \cdots, k$, we  obtain
\begin{equation}
\begin{split}\label{energy-prop1-2}
&  \frac{1}{2}\int_{\mathbb{R}^{2}} \bigl(\varepsilon^{2} |\nabla U_{*}|^{2} + V(x) U_{*}^{2}\bigr) dx -\frac{1}{p} \int_{\mathbb{R}^{2}} |U_{*}|^{p} dx \\
= & \frac12 \sum_{i, j =1}^{k}  \int_{\r^{2}}\big[\Bigl(V(x)- V(x^{0})\Bigl) U_{\va, y^{i}} U_{\va, y^{j}} +  U^{p-1}_{\va, y^{i}} U_{\va, y^{j}}\big]\,dx - \frac {1}{p} \int_{\mathbb{R}^{2}} |U_{*}|^{p} dx.\\
\end{split}
\end{equation}
We write
\begin{eqnarray*}
&&\mathcal{H}:= \sum_{i, j =1}^{k}  \int_{\r^{2}}\Bigl(V(x)- V(x^{0})\Bigl) U_{\va, y^{i}} U_{\va, y^{j}}\nonumber\\
& =&\int_{\r^2}\bigl(V(x)-V(x^0)\bigr)\big(\sik U_{\va, y^i}^2+\sum\limits_{i\neq j}^kU_{\va, y^i}U_{\va, y^j}\big)dx\nonumber\\
&=&\int_{\r^2}\sik \bigl(V(x)-V(y^i)+V(y^i)-V(x^0)\bigr)U_{\va, y^i}^2dx\nonumber\\
&&+\int_{\r^2}\sum\limits_{i\neq j}^k\bigl(V(x)-V(y^i)+V(y^i)-V(x^0)\big)U_{\va,y^i}U_{\va,y^j}dx.\nonumber\\
\end{eqnarray*}
Changing variables we obtain
\begin{eqnarray*}
\mathcal{H}&=&-\va^2\int_{\r^2}U^2(y)\sik \big(V(x^0)-V(y^i)\big)+\va^2\int_{\r^2}\big(V(\va y+y^i)-V(y^i)\big)U^2(y)dy\nonumber\\
&&+\va^2\int_{\r^2}\sum_{i\neq j}^k\big(V(\va y+y^i)-V(y^i)\big)U(y) U(y-\frac{y^i-y^j}{\va})dy\nonumber\\
&&+\va^2\sum_{i\neq j}^k\int_{\r^2}\big(V(x^0)-V(y^i)\big)U(y) U(y-\frac{y^i-y^j}{\va})dy.\nonumber\\
\end{eqnarray*}
By assumption $V_1$ and \eqref{eq1a},
\begin{eqnarray}\label{energy-prop1-2a}
\mathcal{H}&\leq&-\va^2\int_{\r^2}U^2\sik \big(V(x^0)-V(y^i)\big)+\va^2k\int_{\r^2}|\va y|^{\theta}U^2(y)\,dy\nonumber\\
&&+C\va^2\sum_{i\neq j}^k\int_{\r^2}|\va y|^\theta U(y) U(y-\frac{y^i-y^j}{\va})dy++C\va^2\sum_{i\neq j}^k\int_{\r^2}|\va|^\theta U(y) U(y-\frac{y^i-y^j}{\va})dy\nonumber\\
&=&-\va^2\int_{\r^2}U^2\sik \big(V(x^0)-V(y^i)\big)+O(\va^{2+\theta}+\va^{2+\theta}\sum_{i\neq j}^ke^{-|\frac{y^i-y^j}{\va}|}).
\end{eqnarray}
Note that if $p>3,$
\begin{eqnarray*}
&&\int_{\r^2}\sum_{i\neq j}^kU_{\va,y^i}^{p-2}U_{\va,y^j}^2dx\\
&=&C\int_{\r^2}\va^2 \sum_{i\neq j}^kU^{p-2}(y-\frac{y^i}{\va})U^2(y-\frac{y^j}{\va})dy\\
&\leq&C\va^2 \sum_{i\neq j}^ke^{\frac{-\min\{(p-2), 2 \} |y^i-y^j|}{\va}}\leq C\va^2\sum\limits_{i\neq j}^ke^{\frac{-(1+\sigma)|y^i-y^j|}{\va}},
\end{eqnarray*}
and if $2<p\leq 3,$
\begin{eqnarray*}
&&\int_{\r^2}\sum_{i\neq j}^kU_{\va,y^i}^{\frac{p}{2}}U_{\va,y^j}^{\frac{p}{2}}dx\\
&\leq& C\va^2\sum_{i\neq j}^k\int_{\r^2}U^{\frac{p}{2}}U(y-\frac{y^i-y^j}{\va})^{\frac{p}{2}}dy\\
&\leq& C\va^2\sum_{i\neq j}^ke^{\frac{-p|y^i-y^j|}{2\va}}\leq C\va^2\sum\limits_{i\neq j}^ke^{\frac{-(1+\sigma)|y^i-y^j|}{\va}},
\end{eqnarray*}
then we have
\begin{equation}\label{energy-prop1-3}
\begin{array}{ll}
\ds\int_{\r^2}U_{*}^{p}&=\ds\int_{\r^2}\sum_{i=1}^k U_{\va,y^i}^{p}dx+p\ds\int_{\r^2}\sum_{i\neq j}^kU_{\va,y^i}^{p-1}U_{\va,y^j}dx\\
&\quad+\left\{
\begin{array}{ll}
O\Big(\ds\int_{\r^2}\sum_{i\neq j}^{k}U_{\va,y^i}^{p-2}U_{\va,y^j}^2dx\Big),\quad (p>3)\vspace{0.2cm} \\
O\Big(\ds\int_{\r^2}\sum_{i\neq
j}^{k}U_{\va,y^i}^{\frac {p}2}U_{\va,y^j}^{\frac {p}2}\Big),\quad (2<p\leq3)
\end{array}
\right.\vspace{0.2cm}\\
&=k\va^2\ds\int_{\r^2}U^{p}+C p \va^2\sum_{i\neq j}^ke^{\frac{-|y^i-y^j|}{\va}} +O(\va^2\sum\limits_{i\neq j}^ke^{\frac{-(1+\sigma)|y^i-y^j|}{\va}}),
\end{array}
\end{equation}
 $\sigma$ is a small fixed constant.

Inserting \eqref{energy-prop1-2a} and \eqref{energy-prop1-3} into \eqref{energy-prop1-2}, we are led to the result.
\end{proof}

\bigskip

\section{Energy expansion}

\bigskip

In this section, we expand the functional
$$
\mathcal{J}_{\va}({\bf {y}},\varphi)=I_{\va}\big(U_{*} +\varphi\big),\ ({\bf {y}},\varphi)\in
M_{k}^{\va}
$$
and present some basic estimates. The functional $\mathcal{J}_{\va}({\bf {y}}, \varphi)$ is expanded as follows.
\begin{equation}\label{expand}
\mathcal{J}_{\va}({\bf {y}},\varphi)=J_{\va}( {\bf {y}}, 0)+\ell_{\va}(\varphi)+\frac{1}{2}L_{\va}(\varphi)
+R_{\va}(\varphi),
\end{equation}
where
\begin{eqnarray*}
\ell_\va(\varphi)&=&\ds\int_{\r^2}\bigl(V(x)-V(x^0)\bigr)U_{*}\varphi dx+\int_{\r^2}\Bigl(\sik U_{\va ,y^i}^{p-1}\varphi-\big(\sik
U_{\va, y^i}\big)^{p-1}\varphi\Bigr)dx\\
&& + \int_{\r^{2}} \bigl( -\frac{1}{4\pi} \int_{\r^{2}} \frac{x_{2} - y_{2}}{|x- y|^{2}} U^{2}_{*}(y) dy \bigr)^{2} U_{*}(x) \varphi(x) dx \\
&&  + \frac{1}{8\pi^{2}}\int_{\r^{2}}  \bigl( \int_{\r^{2}}  \frac{x_{2} - y_{2}}{|x- y|^{2}} U^{2}_{*}(y) dy \cdot \int_{\r^{2}}  \frac{x_{2} - z_{2}}{|x- z|^{2}} U_{*}(z) \varphi(z) dz \bigr)U_{*}^{2} dx\\
&& + \int_{\r^{2}} \bigl( \frac{1}{4\pi} \int_{\r^{2}} \frac{x_{1} - y_{1}}{|x- y|^{2}} U^{2}_{*}(y) dy \bigr)^{2} U_{*}(x) \varphi(x) dx \\
&&  + \frac{1}{8\pi^{2}}\int_{\r^{2}}  \bigl( \int_{\r^{2}}  \frac{x_{1} - y_{1}}{|x- y|^{2}} U^{2}_{*}(y) dy \cdot \int_{\r^{2}}  \frac{x_{1} - z_{1}}{|x- z|^{2}} U_{*}(z) \varphi(z) dz \bigr)U_{*}^{2} dx,
\end{eqnarray*}

\begin{eqnarray*}
L_{\va}(\varphi) &=& \va^{2} \int_{\r^{2}} |\nabla \varphi|^{2} dx + \ds\int_{\r^2}V(x) \varphi^{2}dx - (p-1) \int_{\r^2}U_{*}^{p-2}\varphi ^{2 }dx\\
&& + \int_{\r^{2}} \bigl( \frac{1}{4\pi} \int_{\r^{2}} \frac{x_{2} - y_{2}}{|x- y|^{2}} U^{2}_{*}(y) dy \bigr)^{2} \varphi^{2}(x) dx \\
&&  + \frac{1}{2\pi^{2}}\int_{\r^{2}}  \bigl( \int_{\r^{2}}  \frac{x_{2} - y_{2}}{|x- y|^{2}} U^{2}_{*}(y) dy \cdot \int_{\r^{2}}  \frac{x_{2} - z_{2}}{|x- z|^{2}} U_{*}(z) \varphi(z) dz \bigr)U_{*}(x) \varphi(x) dx\\
&& + \frac{1}{8\pi^{2}}\int_{\r^{2}} \bigl( \int_{\r^{2}}  \frac{x_{2} - y_{2}}{|x- y|^{2}} U^{2}_{*}(y) dy \cdot \int_{\r^{2}}  \frac{x_{2} - z_{2}}{|x- z|^{2}}  \varphi^{2}(z) dz \bigr)U^{2}_{*}(x)dx\\
&& + \int_{\r^{2}} \bigl( \frac{1}{4\pi} \int_{\r^{2}} \frac{x_{1} - y_{1}}{|x- y|^{2}} U^{2}_{*}(y) dy \bigr)^{2} \varphi^{2}(x) dx \\
&&  + \frac{1}{2\pi^{2}}\int_{\r^{2}}  \bigl( \int_{\r^{2}}  \frac{x_{1} - y_{1}}{|x- y|^{2}} U^{2}_{*}(y) dy \cdot \int_{\r^{2}}  \frac{x_{1} - z_{1}}{|x- z|^{2}} U_{*}(z) \varphi(z) dz \bigr)U_{*}(x) \varphi(x) dx\\
&& + \frac{1}{8\pi^{2}} \int_{\r^{2}} \bigl( \int_{\r^{2}}  \frac{x_{1} - y_{1}}{|x- y|^{2}} U^{2}_{*}(y) dy \cdot \int_{\r^{2}}  \frac{x_{1} - z_{1}}{|x- z|^{2}}  \varphi^{2}(z) dz \bigr)U^{2}_{*}(x)dx\\
&&:=L_{1,\va}(\varphi)+L_{2,\va}(\varphi),
\end{eqnarray*}
where $L_{1,\va}(\varphi)=\va^{2} \int_{\r^{2}} |\nabla \varphi|^{2} dx + \ds\int_{\r^2}V(x) \varphi^{2}dx - (p-1) \int_{\r^2}U_{*}^{p-2}\varphi ^{2 }dx$, $L_{2,\va}(\varphi)$ is the rest,
and

\begin{eqnarray*}
R_{\va}(\varphi)&=&-\frac{1}{p}\ds\int_{\r^2}\Big((U_{*}+\varphi)^{p}-
U_{*}^{p}- pU_{*}^{p-1}\varphi-\frac{1}{2}(p-1)pU_{*}^{p-2}\varphi^2\Big)dx\\
&& + \frac{1}{8\pi^{2}}   \int_{\r^{2}}  \bigl( \int_{\r^{2}}  \frac{x_{2}- y_{2}}{|x - y|^{2}} \varphi^{2}(y)  dy \cdot  \int_{\r^{2}}  \frac{x_{2}- z_{2}}{|x - z|^{2}} U_{*}(z) \varphi(z)  dy \bigr) U^{2}_{*}(x)   dx \\
&& + \frac{1}{4\pi^{2}}  \int_{\r^{2}}  \bigl( \int_{\r^{2}}  \frac{x_{2}- y_{2}}{|x - y|^{2}} U_{*} (y)\varphi(y)  dy\bigr)^{2} U_{*}(x)  \varphi(x) dx  \\
&&  + \frac{1}{8\pi^{2}}  \int_{\r^{2}}  \bigl( \int_{\r^{2}}  \frac{x_{2}- y_{2}}{|x - y|^{2}} U^{2}_{*} (y)  dy \cdot \int_{\r^{2}}  \frac{x_{2}- z_{2}}{|x - z|^{2}} \varphi^{2} (z)  dz  \bigr)U_{*}(x)  \varphi(x) dx   \\
&& + \frac{1}{8\pi^{2}}  \int_{\r^{2}}  \bigl( \int_{\r^{2}}  \frac{x_{2}- y_{2}}{|x - y|^{2}} U^{2}_{*} (y)  dy \cdot \int_{\r^{2}}  \frac{x_{2}- z_{2}}{|x - z|^{2}} U_{*}(z) \varphi(z)  dz  \bigr)  \varphi^{2}(x) dx   \\
&& + \frac{1}{32\pi^{2}}  \int_{\r^{2}}  \bigl( \int_{\r^{2}}  \frac{x_{2}- y_{2}}{|x - y|^{2}} \varphi^{2}(y)  dy   \bigr)^{2} U_{*}^{2}(x) dx   \\
&& + \frac{1}{4\pi^{2}}  \int_{\r^{2}}  \bigl( \int_{\r^{2}}  \frac{x_{2}- y_{2}}{|x - y|^{2}} U_{*} (y)\varphi(y)  dy \cdot \int_{\r^{2}}  \frac{x_{2}- z_{2}}{|x - z|^{2}} \varphi^{2}(z)  dz  \bigr)   U_{*}(x) \varphi(x) dx   \\
&& + \frac{1}{8\pi^{2}}  \int_{\r^{2}}  \bigl( \int_{\r^{2}}  \frac{x_{2}- y_{2}}{|x - y|^{2}} U_{*} (y)\varphi(y)  dy   \bigr)^{2}    \varphi^{2}(x) dx   \\
&& + \frac{1}{16 \pi^{2}}  \int_{\r^{2}}  \bigl( \int_{\r^{2}}  \frac{x_{2}- y_{2}}{|x - y|^{2}} U^{2}_{*} (y)  dy \cdot \int_{\r^{2}}  \frac{x_{2}- z_{2}}{|x - z|^{2}} \varphi^{2}(z)  dz  \bigr)  \varphi^{2}(x) dx   \\
&& + \frac{1}{16 \pi^{2}}  \int_{\r^{2}}  \bigl(  \int_{\r^{2}}  \frac{x_{2}- y_{2}}{|x - y|^{2}} \varphi^{2}(y)  dy  \bigr)^{2} U_{*}(x) \varphi(x) dx \\
&& + \frac{1}{8 \pi^{2}}  \int_{\r^{2}}  \bigl( \int_{\r^{2}}  \frac{x_{2}- y_{2}}{|x - y|^{2}} U_{*} (y) \varphi(y)  dy \cdot \int_{\r^{2}}  \frac{x_{2}- z_{2}}{|x - z|^{2}} \varphi^{2}(z)  dz  \bigr)  \varphi^{2}(x) dx   \\
&& + \frac{1}{32 \pi^{2}}  \int_{\r^{2}}  \bigl(  \int_{\r^{2}}  \frac{x_{2}- y_{2}}{|x - y|^{2}} \varphi^{2}(y)  dy  \bigr)^{2} \varphi^{2}(x) dx \\
&& + \frac{1}{8\pi^{2}}   \int_{\r^{2}}  \bigl( \int_{\r^{2}}  \frac{x_{1}- y_{1}}{|x - y|^{2}} \varphi^{2}(y)  dy \cdot  \int_{\r^{2}}  \frac{x_{1}- z_{1}}{|x - z|^{2}} U_{*}(z) \varphi(z)  dy \bigr) U^{2}_{*}(x)   dx \\
&& + \frac{1}{4\pi^{2}}  \int_{\r^{2}}  \bigl( \int_{\r^{2}}  \frac{x_{1}- y_{1}}{|x - y|^{2}} U_{*} (y)\varphi(y)  dy\bigr)^{2} U_{*}(x)  \varphi(x) dx  \\
&&  + \frac{1}{8\pi^{2}}  \int_{\r^{2}}  \bigl( \int_{\r^{2}}  \frac{x_{1}- y_{1}}{|x - y|^{2}} U^{2}_{*} (y)  dy \cdot \int_{\r^{2}}  \frac{x_{1}- z_{1}}{|x - z|^{2}} \varphi^{2} (z)  dz  \bigr)U_{*}(x)  \varphi(x) dx   \\
&& + \frac{1}{8\pi^{2}}  \int_{\r^{2}}  \bigl( \int_{\r^{2}}  \frac{x_{1}- y_{1}}{|x - y|^{2}} U^{2}_{*} (y)  dy \cdot \int_{\r^{2}}  \frac{x_{1}- z_{1}}{|x - z|^{2}} U_{*}(z) \varphi(z)  dz  \bigr)  \varphi^{2}(x) dx   \\
&& + \frac{1}{32\pi^{2}}  \int_{\r^{2}}  \bigl( \int_{\r^{2}}  \frac{x_{1}- y_{1}}{|x - y|^{2}} \varphi^{2}(y)  dy   \bigr)^{2} U_{*}^{2}(x) dx   \\
&& + \frac{1}{4\pi^{2}}  \int_{\r^{2}}  \bigl( \int_{\r^{2}}  \frac{x_{1}- y_{1}}{|x - y|^{2}} U_{*} (y)\varphi(y)  dy \cdot \int_{\r^{2}}  \frac{x_{1}- z_{1}}{|x - z|^{2}} \varphi^{2}(z)  dz  \bigr)   U_{*}(x) \varphi(x) dx   \\
&& + \frac{1}{8\pi^{2}}  \int_{\r^{2}}  \bigl( \int_{\r^{2}}  \frac{x_{1}- y_{1}}{|x - y|^{2}} U_{*} (y)\varphi(y)  dy   \bigr)^{2}    \varphi^{2}(x) dx   \\
&& + \frac{1}{16 \pi^{2}}  \int_{\r^{2}}  \bigl( \int_{\r^{2}}  \frac{x_{1}- y_{1}}{|x - y|^{2}} U^{2}_{*} (y)  dy \cdot \int_{\r^{2}}  \frac{x_{1}- z_{1}}{|x - z|^{2}} \varphi^{2}(z)  dz  \bigr)  \varphi^{2}(x) dx   \\
&& + \frac{1}{16 \pi^{2}}  \int_{\r^{2}}  \bigl(  \int_{\r^{2}}  \frac{x_{1}- y_{1}}{|x - y|^{2}} \varphi^{2}(y)  dy  \bigr)^{2} U_{*}(x) \varphi(x) dx \\
&& + \frac{1}{8 \pi^{2}}  \int_{\r^{2}}  \bigl( \int_{\r^{2}}  \frac{x_{1}- y_{1}}{|x - y|^{2}} U_{*} (y) \varphi(y)  dy \cdot \int_{\r^{2}}  \frac{x_{1}- z_{1}}{|x - z|^{2}} \varphi^{2}(z)  dz  \bigr)  \varphi^{2}(x) dx   \\
&& + \frac{1}{32 \pi^{2}}  \int_{\r^{2}}  \bigl(  \int_{\r^{2}}  \frac{x_{1}- y_{1}}{|x - y|^{2}} \varphi^{2}(y)  dy  \bigr)^{2} \varphi^{2}(x) dx\\
&:=&R_{1,\va}(\varphi)+R_{2,\va}(\varphi),
\end{eqnarray*}
where $R_{1,\va}(\varphi)=-\frac{1}{p}\ds\int_{\r^2}\Big((U_{*}+\varphi)^{p}-
U_{*}^{p}- pU_{*}^{p-1}\varphi-\frac{1}{2}(p-1)pU_{*}^{p-2}\varphi^2\Big)dx$ and $R_{2,\va}(\varphi)=R_{\va}(\varphi)-R_{1,\va}(\varphi)$.

In order to find a critical point $({\bf {y}},\varphi)\in M_k^{\va}$ for $J_{\va}({\bf {y}},\varphi),$ we need
to estimate each term in expansion \eqref{expand}.

\begin{lemma}\label{lm3.1}
There exists a constant $C> 0$, such that
\begin{equation*}
\begin{split}
\|R^{(i)}_{\va}(\varphi)\| &\leq C \va^{- \min\{1, p-2\}} \|\varphi\|_{\va}^{\min\{3-i, p-i\}}\\
 &+  C\Bigl( \va \|\varphi\|_{\va}^{3- i } + \|\varphi\|^{4-i }_{\va} + \va^{-1} \|\varphi\|^{5-i }_{\va}  + \va^{-2} \|\varphi\|^{6-i }_{\va} \Bigr).\\
 \end{split}
\end{equation*}
\end{lemma}
\begin{proof} First we estimate $R_{1,\va}$. Let $\tilde{\varphi}=\varphi(\va x).$ We have
\begin{eqnarray}\label{1m3.1-0}
\int_{\r^2}|\varphi|^pdx&=&
\va^2\int_{\r^2}|\widetilde{\varphi}|^p\leq
C\va^2\bigl(\int_{\r^2}(|\nabla\widetilde{\varphi}|^2+|\widetilde{\varphi}|^2)dx\bigr)^{\frac p2}\nonumber\\
&=&C\va^2\bigl(\va^{-2}\int_{\r^2}(\va^2|\nabla
\varphi|^2+|\varphi|^2)dx\bigr)^{\frac p2}\nonumber\\
&\leq& C\va^{2-p}\|\varphi\|^p_{\va}.
\end{eqnarray}

If $2<p\leq3$, we find
\begin{equation*}\label{1m3.1-1}
\Big|R_{1,\va}(\varphi)\Big|\leq C\int_{\r^2}|\varphi|^{p}dx\leq C\va^{2-p}\|\varphi\|_{\va}^{p},
\end{equation*}
\begin{eqnarray*}\label{1m3.1-2}
\Big|\langle R'_{1,\va}(\varphi),\psi\rangle\Big|&\leq&
C\int_{\r^2}|\varphi|^{p-1}\psi dx \nonumber\\
&\leq&C\Big(\int_{\r^2}|\varphi|^{p}dx\Bigr)^{\frac{p-1}{p}}\Bigl(\int_{\r^2}|\psi|^{p}dx\Bigr)^{\frac{1}{p}}\nonumber\\
&\leq&C\Bigl(\va^{2-p}\|\varphi\|_{\va}^{p}\Big)^{\frac{p-1}{p}}
\Big(\va^{2-p}\|\psi\|_{\va}^{p}\Big)^{\frac{1}{p}}\nonumber\\
&\leq&C\va^{2-p}\|\varphi\|_{\va}^{p-1}\|\psi\|_{\va}
\end{eqnarray*}
and
\begin{eqnarray*}\label{1m3.1-3}
\Big|\langle R''_{1,\va}(\varphi)(\psi,\xi)\rangle\Big|&\leq&
C\int_{\r^2}|\varphi|^{p-2}|\psi||\xi| dx\nonumber\\
&\leq&C\Big(\int_{\r^2}|\varphi|^{(p-2)\frac{p}{p-2}}dx\Big)^{\frac{p-2}{p}}\Big(\int_{\r^2}|\psi|^{p}dx\Big)^{\frac{1}{p}}
\Big(\int_{\r^2}|\xi|^{p}dx\Big)^{\frac{1}{p}}\nonumber\\
&\leq& C\Big(\va^{2-p}\|\varphi\|_{\va}^{p}\Big)^{\frac{p-2}{p}}\Big(\va^{2-p}\|\psi\|_{\va}^{p}\Big)
^{\frac{1}{p}}\Big(\va^{2-p}\|\xi\|_{\va}^{p}\Big)^{\frac{1}{p}}\nonumber\\
&\leq&C\va^{2-p}\|\varphi\|_{\va}^{p-1}\|\psi\|_{\va}\|\xi\|_{\va}.
\end{eqnarray*}

If $p>3$, we estimate
$$
\Big|R_{1,\va}(\varphi)\Big|\leq C\int_{\r^2}U_{\va,y}^{p-1}|\varphi|^3\leq C\va^{-1}\|\varphi\|_{\va}^3,
$$
$$
\Big|\langle R'_{1,\va}(\varphi),\psi\rangle\Big|\leq C\va^{-1}\|\varphi\|_{\va}^2\|\psi\|_{\va}
$$
and
$$
\Big|\langle R''_{1,\va}(\varphi)(\psi,\xi)\rangle\big|\leq C\va^{-1}\|\varphi\|_{\va}\|\psi\|_{\va}\|\xi\|_{\va}.
$$

Next, we  estimate $R_{2,\va}$. We commence with some fundamental estimates needed. By  H\"{o}lder inequality and \eqref{1m3.1-0}, we have
\begin{equation}\label{lm3.1-4}\begin{split}
| \int_{\r^{2}} \frac{x_{2}- y_{2}}{|x - y|^{2}} U_{*} (y)\varphi(y)  dy|&\leq |  \int_{\r^{2}}  \frac{1}{|x - y|} U_{*} (y)\varphi(y)  dy|\\
\leq &\bigl(\int_{\r^{2}} \frac{1 }{|x - y|^{\frac32}}U_{*}^{\frac32} (y) dy \bigr)^{\frac23} \bigl(\int_{\r^{2}} |\varphi|^{3}\bigr)^{\frac13}\\
\leq & C\va^{-\frac13}\|\varphi\|_{\va}\big(\int_{\r^{2}}\va^2\sum\frac{U^{\frac32}(z)}{|x-(\va z+y^i)|^{\frac32}}dz\big)^{\frac23}\\
\leq & C\va^{-\frac13}\|\varphi\|_{\va}\big(\int_{\r^{2}}\frac{\va^2}{\va^{\frac32}}\sum\frac{U^{\frac32}(z)}{|z-\frac{x-y^i}{\va}|^{\frac32}}dz\big)^{\frac23}\\
\leq & C \|\varphi\|_{\va}
   \end{split}
 \end{equation}
 and
\begin{equation}\label{lm3.1-5}
\begin{split}
\big|\int_{\r^{2}} \frac{x_{1} -z_{1}}{|x -z|^{2}} \varphi^{2} dz \big|\leq  & \int_{B_{\va}(x)}  \frac{1}{|x -z|} \varphi^{2} dz  + \int_{\r^{2} \setminus B_{\va}(x)}  \frac{1}{|x -z|} \varphi^{2} dz \\
\leq  & C \bigl( \int_{B_{\va}(x)} \frac{1}{|x -z|^{\frac32}}\,dz\bigr)^{\frac23}  \|\varphi\|_{L^{6}(\r^{2})}^{2} +  \va^{-1} \|\varphi\|^{2}_{L^{2}(\r^{2})} \\
\leq & C \bigl(\va^{\frac13} \va^{-\frac43} \|\varphi\|_{\va}^{2} + \va^{-1} \|\varphi\|_{L^{2}(\r^{2})}^{2}\bigr)\\
 \leq &  C \va^{-1} \|\varphi\|_{\va}^{2}.
\end{split}
\end{equation}
Equations \eqref{lm3.1-4}, \eqref{lm3.1-5} and H\"{o}lder inequality yield
 \begin{equation*}\begin{split}
& | \int_{\r^{2}}  \bigl( \int_{\r^{2}}  \frac{x_{2}- y_{2}}{|x - y|^{2}} \varphi^{2}(y)  dy \cdot  \int_{\r^{2}}  \frac{x_{2}- z_{2}}{|x - z|^{2}} U_{*}(z) \varphi(z)  dy \bigr) U^{2}_{*}(x)   dx | \\
\leq & C\va^{-1} \|\varphi\|^2_{\va} \|\varphi\|_{\va}\int_{\r^{2}} U^{2}_{*}(x)dx \leq C \va \|\varphi\|_{\va}^{3}
\end{split}
 \end{equation*}

and
\begin{equation*}\begin{split}
  &\big| \int_{\r^{2}}  \bigl( \int_{\r^{2}}  \frac{x_{2}- y_{2}}{|x - y|^{2}} U_{*} (y)\varphi(y)  dy\bigr)^{2} U_{*}(x)  \varphi(x) dx\big| \\
  \leq & C \|\varphi\|^{2}_{\va} \bigl( \int_{\r^{2}} U_{*}^{2} dx \bigr)^{\frac 12 } \bigl( \int_{\r^{2}} \varphi^{2} dx \bigr)^{\frac 12 } \leq  C  \va \|\varphi\|_{\va}^{3}.
\end{split}
\end{equation*}
Similarly, we have the following estimates:
\begin{eqnarray*}
&& \big| \int_{\r^{2}}  \bigl( \int_{\r^{2}}  \frac{x_{2}- y_{2}}{|x - y|^{2}} U^{2}_{*} (y)  dy \cdot \int_{\r^{2}}  \frac{x_{2}- z_{2}}{|x - z|^{2}} \varphi^{2} (z)  dz  \bigr)U_{*}(x)  \varphi(x)\, dx\big|\\
&\leq &
C\va^{-1} \|\varphi\|_{\va}^{2}  \int_{\r^{2}} \int_{\r^{2}} \frac{1}{|x - y|} U^{2}(y^j)  dy U_{y^{i}}(x)  \varphi(x) dx
\leq  C \va \|\varphi\|_{\va}^{3},
\end{eqnarray*}
\begin{eqnarray*}
 && \bigg|\int_{\r^{2}}  \bigl( \int_{\r^{2}}  \frac{x_{2}- y_{2}}{|x - y|^{2}} U^{2}_{*} (y)  dy \cdot \int_{\r^{2}}  \frac{x_{2}- z_{2}}{|x - z|^{2}} U_{*}(z) \varphi(z)  dz  \bigr)  \varphi^{2}(x) dx\bigg| \\
&\leq  & C \|\varphi\|_{\va} \int_{\r^{2}}  \bigl( \int_{\r^{2}}  \frac{1}{|x - y|} U^{2}(\frac{y-y^{i}}{\va})  dy   \bigr)  \varphi^{2}(x) dx \leq  C	 \va  \|\varphi\|_{\va} \int_{\r^{2}}\varphi^{2}(x) dx\\
&\leq & C	 \va \|\varphi\|_{\va}^{3},
\end{eqnarray*}

\begin{equation*}\begin{split}
&| \int_{\r^{2}}  \bigl( \int_{\r^{2}}  \frac{x_{2}- y_{2}}{|x - y|^{2}} \varphi^{2}(y)  dy   \bigr)^{2} U_{*}^{2}(x) dx |\leq  C\va^{-2}\|\varphi\|^{4}_{\va}\int_{\r^{2}} U_{*}^{2}(x) dx\leq C\|\varphi\|^{4}_{\va},
\end{split}
\end{equation*}

\begin{equation*}\begin{split}
&  |\int_{\r^{2}}  \bigl( \int_{\r^{2}}  \frac{x_{2}- y_{2}}{|x - y|^{2}} U_{*} (y)\varphi(y)  dy \cdot \int_{\r^{2}}  \frac{x_{2}- z_{2}}{|x - z|^{2}} \varphi^{2}(z)  dz  \bigr)   U_{*}(x) \varphi(x) dx | \\
\leq & C\va^{-1} \|\varphi\|^3_{\va}\int_{\r^{2}} |U_{*}(x) \varphi(x)| dx\leq   C \|\varphi\|_{\va}^{4},
\end{split}
\end{equation*}

$$ \int_{\r^{2}}  \bigl( \int_{\r^{2}}  \frac{x_{2}- y_{2}}{|x - y|^{2}} U_{*} (y)\varphi(y)  dy   \bigr)^{2}    \varphi^{2}(x) dx
\leq  C \|\varphi\|_{\va}^{2} \int_{\r^{2}}  \varphi^{2}(x) dx \leq C \|\varphi\|_{\va}^{4}.
$$

\[
\begin{split}
 &\big|\int_{\r^{2}}  \bigl( \int_{\r^{2}}  \frac{x_{2}- y_{2}}{|x - y|^{2}} U^2_{*} (y)\,dy \cdot \int_{\r^{2}}  \frac{x_{2}- z_{2}}{|x - z|^{2}} \varphi^{2}(z)  dz  \bigr)  \varphi^{2}(x) dx \big|\\ &\leq C \|\varphi\|_{\va}^{2}\int_{\r^{2}}\varphi^{2}(x) dx\leq C  \|\varphi\|_{\va}^{4},\\
 \end{split}
\]

\[
\begin{split}
&\big|\int_{\r^{2}}  \bigl(  \int_{\r^{2}}  \frac{x_{2}- y_{2}}{|x - y|^{2}} \varphi^{2}(y)  dy  \bigr)^{2} U_{*}(x) \varphi(x)\, dx\big|\\
\leq C &\va^{-2} \|\varphi\|_{\va}^{4}  \int_{\r^{2}} U_{*}(x) \varphi(x) dx
\leq   C \va^{-1} \|\varphi\|_{\va}^{5},\\
\end{split}
\]

\[
\begin{split}
&\big|\int_{\r^{2}}\bigl( \int_{\r^{2}}  \frac{x_{2}- y_{2}}{|x - y|^{2}} U_{*}(y) \varphi(y)  dy \cdot \int_{\r^{2}}\frac{x_{2}- z_{2}}{|x - z|^{2}} \varphi^{2}(z)  dz \bigr)\varphi^{2}(x)dx\big|\\
&\leq C\va^{-1}\|\varphi\|^{3}_{\va}\int_{\r^{2}} \varphi^{2}(x) dx\leq C\va^{-1} \|\varphi\|^{5}_{\va},\\
\end{split}
\]
and
$$
 \int_{\r^{2}}  \bigl(  \int_{\r^{2}}  \frac{x_{1}- y_{1}}{|x - y|^{2}} \varphi^{2}(y)  dy  \bigr)^{2} \varphi^{2}(x) dx \leq  C \va^{-2} \|\varphi\|_{\va}^{4}  \int_{\r^{2}} \varphi^{2}(x) dx
 \leq   C \va^{-2} \|\varphi\|_{\va}^{6} .
$$
The rest terms in  $R_{2,\va}$ can be estimated in the same way. The assertion follows by putting all these estimates together.
\end{proof}

\bigskip

\begin{lemma}\label{lm3.2}
There holds
\begin{equation}
\|\ell_{\va}\| \leq C \va^{1+\theta} \|\varphi\|_{\va} + C \sum_{i=1}^{k} \va | V( y^{i})-V(x^{0}) | + C\va e^{-\min\{\frac {p-1}2,1\}\frac{|y^i-y^j|}{\va}}\|\varphi\|_{\va} .
\end{equation}
\end{lemma}
\begin{proof}
By assumption $(V_1)$, one has
\begin{equation}\label{lm3.2-1}\begin{split}
& \big|\int_{\r^2}\bigl(V(x)-V(x^0)\bigr)U_{*}\varphi dx\big|
=  \big|\sum_{i=1}^{k} \int_{\r^{2}}  \bigl(V(x)-V(x^0)\bigr) U_{\va, y^{i}} \varphi\,dx\big| \\
= & \sum_{i=1}^{k} \int_{\r^{2}}  \bigl(V(x)-V(y^{i})\bigr) U_{\va, y^{i}} \varphi\,dx  + \sum_{i=1}^{k} \int_{\r^{2}}  \bigl(V(y^{i})-V(x^0)\bigr) U_{\va, y^{i}} \varphi\,dx\big| \\
\leq &  \sum_{i=1}^{k} \Bigl( \int_{\r^{2}} [\bigl(V(x)-V(y^{i})\bigr) U_{\va, y^{i}}]^{2} dx\Bigr)^{\frac12} \|\varphi\|_{\va}\\
 &+   \sum_{i=1}^{k} \Bigl( \int_{\r^{2}} [\bigl(V(y^{i})-V(x^{0})\bigr) U_{\va, y^{i}}]^{2} dx\Bigr)^{\frac12} \|\varphi\|_{\va} \\
= & \sum_{i=1}^{k} \va \Bigl( \int_{\r^{2}} [\bigl(V(\va x + y^{i} ) -V(y^{i})\bigr) U(x) ]^{2} dx\Bigr)^{\frac12} \|\varphi\|_{\va} \\
&+   \sum_{i=1}^{k} \va | V( y^{i})-V(x^{0}) | \Bigl( \int_{\r^{2}} U^{2}(x) dx\Bigr)^{\frac12} \|\varphi\|_{\va} \\
\leq & C \va^{1+\theta} \|\varphi\|_{\va} + C \sum_{i=1}^{k} \va | V( y^{i})-V(x^{0}) |.
\end{split}
\end{equation}
On the other hand, we have
\begin{eqnarray}\label{lm3.2-2}
&&\Big|\int_{\r^2}\sik U_{\va, y^i}^{p-1}\varphi-\Big(\sik U_{\va, y^i}\Big)^{p-1}\varphi dx\Big|
=\left\{\begin{array}{ll}
       O\Bigl(\ds\int_{\r^N}\sum\limits_{i\neq j}U_{\va,y^i}^{p-2}U_{\va,y^j}\varphi dx\Bigr),\ \ \ & \hbox{$\text{if}\,\,p>3$},\vspace{2mm}\nonumber\\
        O\Bigl(\ds\int_{\r^N}\sum\limits_{i\neq j}U_{\va,y^i}^{\frac{p-1}{2}}U_{\va,y^j}^{\frac{p-1}{2}}\varphi dx\Bigr),\ \ \ & \hbox{$\text{if}\,\,2<p\leq3$}\nonumber\\
    \end{array}\right.\\
&\leq&C\va e^{-\min\{\frac {p-1}2,1\}\frac{|y^i-y^j|}{\va}}\|\varphi\|_{\va}.
\end{eqnarray}
As the proof of  Lemma \ref{lm3.1}, we may verify
\begin{eqnarray}\label{lm3.2-3}
&& \big|\int_{\r^{2}} \bigl(  \int_{\r^{2}} \frac{x_{2} - y_{2}}{|x- y|^{2}} U^{2}_{*}(y) dy \bigr)^{2} U_{*}(x) \varphi(x) dx+\int_{\r^{2}} \bigl(  \int_{\r^{2}} \frac{x_{1} - y_{1}}{|x- y|^{2}} U^{2}_{*}(y) dy \bigr)^{2} U_{*}(x) \varphi(x) dx\big|\nonumber\\
 & \leq &  C \va^{2} \big|\int_{\r^{2}} U_{*}(x) \varphi(x) dx\big|\leq C\va^{3} \|\varphi\|_{\va},
\end{eqnarray}
and
\begin{eqnarray}\label{lm3.2-4}
&& |\int_{\r^{2}}  \bigl( \int_{\r^{2}}  \frac{x_{2} - y_{2}}{|x- y|^{2}} U^{2}_{*}(y) dy \cdot \int_{\r^{2}}  \frac{x_{2} - z_{2}}{|x- z|^{2}} U_{*}(z) \varphi(z) dz \bigr)U_{*}^{2} dx|\nonumber\\
&&+|\int_{\r^{2}}  \bigl( \int_{\r^{2}}  \frac{x_{1} - y_{21}}{|x- y|^{2}} U^{2}_{*}(y) dy \cdot \int_{\r^{2}}  \frac{x_{2} - z_{2}}{|x- z|^{2}} U_{*}(z) \varphi(z) dz \bigr)U_{*}^{2} dx|\nonumber\\
&\leq& C \va \|\varphi\|_{\va}\int_{\r^{2}}  U^{2}_{*}(x)dx\leq C\va^{3} \|\varphi\|_{\va}.
\end{eqnarray}
So the result follows from \eqref{lm3.2-1}-\eqref{lm3.2-4}.
\end{proof}
\bigskip

Since $L_{\va}$ is bounded and bi-linear operator, the following lemma can be obtained directly.

\begin{lemma}
There exists a positive constant, independent of $\va$, such that
$$
\|L_{\va}v\|\leq C\|v\|, \ \ v\in E.
$$
\end{lemma}

\bigskip
Next, we show that $L_{\va}$ is invertible in $E$.
\begin{lemma}\label{lm3.4} There are positive constants $\va_0$ and $\mu_0$, such that for any $0<\va<\va_0$
and $v\in E$,
\begin{equation}\label{lm3.4-0}
\|L_{\va}v\|\geq \mu_0\|v\|.
\end{equation}
\end{lemma}
\begin{proof}
We argue indirectly. Suppose on the contrary that there exist $\va_n\rightarrow0$, $v_n\in E_{n}$ and ${\bf {y} }^{ n}=(y^{1,n},\cdots,y^{k,n})\in D_k^{\va_n,\delta}$ such that
\begin{equation}\label{lm3.4-1}
\langle L_{\va_n}v_n,\varphi\rangle=o_n(1)\|v_n\|_{\va_n}\|\varphi\|_{\va_n},\ \ \ \ \,\,\forall\,\varphi\in E_{n}.
\end{equation}
Without loss of generality, we assume that $\|v_n\|_{\va_n}=\va_n$. For any $\varphi \in E$, we have
\begin{equation}\begin{split}
 \bigg|\int_{\r^{2}} \bigl( \frac{1}{4\pi} \int_{\r^{2}} \frac{x_{2} - y_{2}}{|x- y|^{2}} U^{2}_{*}(y) \,dy \bigr)^{2} v_{n} \varphi  \,dx\bigg| \leq C \va_{n}^{2} \|v_{n}\|_{\va_{n}} \|\varphi\|_{\va_{n}}
\end{split}
\end{equation}
and
\begin{equation}\begin{split}
& \frac{1}{4\pi^{2}}\bigg|\int_{\r^{2}}  \bigl( \int_{\r^{2}}  \frac{x_{2} - y_{2}}{|x- y|^{2}} U^{2}_{*}(y) dy \cdot \int_{\r^{2}}  \frac{x_{2} - z_{2}}{|x- z|^{2}} U_{*}(z)v_{n} dz \bigr)U_{*}(x) \varphi(x) dx\bigg|\\
\leq & C \va_{n} \int_{\r^{2}} \int_{\r^{2}}  \frac{1}{|x- z|} U_{*}(z)|v_{n}| dz \bigr)U_{*}(x) |\varphi(x)| dx \\
\leq & C \va_{n} \|v_{n}\|_{\va_{n}} \int_{\r^{2}}U_{*}(x) \varphi(x) dx \\
\leq & C \va_{n}^{2} \|v_{n}\|_{\va_{n}} \|\varphi\|_{\va_{n}}
\end{split}
\end{equation}
Similarly, we get
\[ \big|\langle L_{2,\va_{n}} v_{n}, \varphi \rangle\big| \leq C \va^{2}_{n} \|v_{n}\|_{\va_{n}} \|\varphi\|_{\va_{n}}, \]
which implies
\begin{equation}\label{lm3.4-2}
\langle L_{1,\va_{n}} v_{n}, \varphi \rangle = o(1)\|v_{n}\|_{\va_{n}} \|\varphi\|_{\va_{n}} , \text{ for } \varphi \in E_{n}.
 \end{equation}
Fix $i\in\{1,\cdots,k\}$ and let
 $$
 \tilde{v}_{n,i}(x)=v_n(\va_n x+y^{i,n}).
 $$
 We have
$$
\va_n^{2}\int_{\r^{2}}|\nabla v_{\va}|^{2} dx  +\int_{\r^2}V(x)v_n^2dx=O(\va_n^2),
$$
and
\begin{equation}\label{lm3.4-3}
\int_{\r^{2}}|\tilde{v}_{n,i}|^{2} dx +\int_{\r^2}V(\va_nx+y^{i,n})\tilde{v}_{n,i}dx\leq C.
\end{equation}
Then it follows from \eqref{lm3.4-3} that $\{\tilde{v}_{n,i}\}$ is bounded in $H^1(\r^2)$. Thus, there is a subsequence still denote by $\tilde{v}_{n,i}$,
such that for $n\rightarrow\infty,$
$$
\tilde{v}_{n,i} \rightharpoonup v_i\ \ \ \hbox{in}\,\ H^1(\r^2),\,\,\,\,\,\ \
\tilde{v}_{n,i} \rightarrow v_i\ \ \ \hbox{in}\,\ L_{loc}^{p}(\r^2),\ 2\leq p< +\infty.
$$

Now we claim that $v_{i} = 0 $. For any $\psi \in H^{1}(\r^{N})$, we define
\begin{equation}\label{non-3}
P_{\va_{n}}\psi = \psi - \sum_{l=1}^{2}\sum_{i=1}^{k} \alpha_{\va_{n, l , i}}\frac{\partial U_{\va_{n}, y^{i}}}{\partial y^{i}_{l}} \in E_{n},
\end{equation}
where
\[\alpha_{\va_{n}, l, i} = \sum_{m=1}^{2} \sum_{j=1}^{k} C_{\va_{n}, l, i}^{m, j}  \langle  \frac{\partial U_{\va_{n}, y^{j}}}{\partial y^{j}_{m}}, \psi \rangle_{\va_{n}},  \]
for some constants $C_{\va_{n}, l, i}^{m, j}$.
Let
\[\beta_{\va_{n}, l, i} = \langle L_{\va_{n}, 1} v_{n}, \frac{\partial U_{\va_{n}, y^{i}}}{\partial y^{i}_{l}} \rangle_{\va_{n}}. \]
Then, we have
\begin{equation}\label{lm3.4-4}\begin{split}
 \langle L_{1,\va_{n}} v_{n}, \psi \rangle_{\va_{n}}
= & \langle L_{\va_{n}, 1} v_{n}, P_{\va_{n}}\psi \rangle_{\va_{n}} + \sum_{l=1}^{2}\sum_{i=1}^{k}  \alpha_{\va_{n, l , i}}\beta_{\va_{n}, l, i} \\
= & o(1)  \|v_{n}\|_{\va_{n}} \|\psi\|_{\va_{n}} + \sum_{l=1}^{2}\sum_{i=1}^{k}  \alpha_{\va_{n, l , i}}\beta_{\va_{n}, l, i}\\
= & o(1) \|v_{n}\|_{\va_{n}}  \|\psi\|_{\va_{n}} + \sum_{l=1}^{2}\sum_{i=1}^{k}  \gamma_{\va_{n, l , i}} \langle  \frac{\partial U_{\va_{n}, y^{i}}}{\partial y^{i}_{l}}, \psi \rangle_{\va_{n}},
\end{split}
 \end{equation}
 where $\gamma_{\va_{n, l , i}} =  \sum_{m=1}^{2} \sum_{j=1}^{k} C_{\va_{n}, m, j}^{l, i}  \beta_{\va_{n}, m, j} . $
 Choosing $\varphi =\frac{\partial U_{\va_{n}, y^{i}}}{\partial y^{i}_{l}} $ in \eqref{lm3.4-4}, we can estimate
 \[\gamma_{\va_{n, l , i}} = o(\va_{n}). \]
 Hence, \eqref{lm3.4-4} becomes
 \begin{equation}\label{lm3.4-5}\begin{split}
 \langle L_{1, \va_{n}} v_{n}, \psi \rangle_{\va_{n}}
= & o(1) \|v_n\|_{\va_{n}}   \|\psi\|_{\va_{n}}, \, \, \forall \psi \in H^{1}(\r^{2}).
\end{split}
 \end{equation}
 Let $\bar{\varphi}_{n}(x) = \varphi(\frac{x - y^{i, n}}{\va}) $ and substitute it into \eqref{lm3.4-5},
we obtain
\begin{equation*}\begin{split}
 & \int_{\r^{2}} \nabla \tilde{v}_{n , i} \nabla \varphi dx + \ds\int_{\r^2}V(\va_{n} x + y^{i, n})  \tilde{v}_{n , i}  \varphi dx - (p-1) \int_{\r^2}U_{*}(\va x + y^{i, n})^{p-2} \tilde{v}_{n , i}  \varphi   dx\\
 &=   \va_{n}^{-2} \Bigl\{ \va_{n}^{2} \int_{\r^{2}} \nabla v_{n} \nabla \bar{\varphi}_{n}   dx + \ds\int_{\r^2}V(x)  v_{n} \bar{\varphi}_{n}  dx - (p-1) \int_{\r^2}U_{*}^{p-2} v_{n} \bar{\varphi}_{n} dx\Bigr\}\\
&=  \va_{n}^{-2} o(\va_{n})  \|\bar{\varphi}_{n}\|_{\va_{n}} = o(1) \|\varphi\|.
\end{split}
\end{equation*}
Therefore, $v_{i}$ satisfies the equation
\[ -\Delta v_{i} + V(y^{i}) v_{i} = (p-1) U^{p-2}v_{i}, \]
and the non-degeneracy of the solution $U$ gives
\[v_{i} = \sum_{i=1}^{2} c_{i} \frac{\partial U}{\partial x_{i}}. \]
Since $v_{n} \in E_{n}$, that is,
\begin{equation*}\label{non-2}
 \langle v_{n}, \frac{\partial U_{\va_{n}, y_{l}^{i}}}{\partial y_{l}^{i}} \rangle _{\va_{n}}=0,
\end{equation*}
we deduce that
\[ \langle v_{i}, \frac{\partial U}{\partial x_{l}} \rangle= 0, l =1, 2. \]
which gives $c_{1} = c_{2} = 0$, and then $v_i=0.$ As a result,
\begin{eqnarray*}
o_n(1)\va_n^2
&=& \langle L_{\va_{n}} v_{n}, v_{n}\rangle =   \langle L_{\va_{n}, 1} v_{n}, v_{n}\rangle +  \langle L_{\va_{n}, 2 } v_{n}, v_{n}\rangle \\
& = & \|v_{n}\|^{2}_{\va_{n}} -   (p-1) \int_{\r^2}U_{*}^{p-2} v_{n}^{2}  dx + o(1)\|v_{n}\|^{2}_{\va_{n}} \\
& =&(1 + o(1)) \va_n^2 - (p-1) \int_{B_{R}(0)} U_{*}^{p-2} v_{n}^{2}  dx - (p-1) \int_{\r^{2} \setminus B_{R}(0)} U_{*}^{p-2} v_{n}^{2}  dx \\
& = & (1 + o(1)+ o_{R}(1))\va_{n}^{2},
\end{eqnarray*}
which is impossible for large $n$. Hence, the conclusion follows.
\end{proof}

\bigskip

\section{Proof of the main result}

\bigskip

In this section, we prove Theorem \ref{th2} by the reduction method.

\begin{proposition}\label{prop3.2} For $\va$ sufficiently small, there is a $C^1$ map from $D_k^{\va,\delta}$ to
$E$ and
$$
 J_\va'(\omega)\Big|_E=0.
$$
Moreover, there exists a constant $C>0$ independent of $\va$ small enough such that
\begin{equation*}
\|\omega\|_{\va}\leq  C\big(\va^{\frac{N}{2}+\min\{\theta,2\}}+\va^{\frac{N}{2}} \sum\limits_{i=1}^k (V(y^i)-V(x^0))+\va^{\frac{N}{2}}e^{-\min\{\frac p2,1\}\frac{|y^i-y^j|}{\va}}\big).
\end{equation*}
\end{proposition}

\begin{proof}
We will use the contraction theorem to prove it.
By the Lemma \ref{lm3.2}, $\ell(\omega)$ is a bounded linear
functional in $E$. The Riesz representation theorem implies that  there is
an $\bar{\ell}_\varepsilon\in E,$ such that
$$
\ell_\varepsilon(\omega)=\langle \bar{\ell}_\varepsilon,\omega\rangle.
$$
Therefore, finding a critical point for $J(\omega)$ is equivalent to
solving
\begin{equation}\label{prop3.2-1}
\bar{\ell}_\va+L_\va(\omega)+R'_\va(\omega)=0.
\end{equation}
By Lemma \ref{lm3.4}, $L_\va$ is invertible. Thus \eqref{prop3.2-1}
is equivalent to
$$
\omega=A(\omega):=-L_\varepsilon^{-1}(\bar{\ell}_\varepsilon+R_\varepsilon'(\omega)).
$$

We set
$$
S_{\va}:=\big\{\omega\in E:\ \ \|\omega\|_{\va}\leq \va^{1+\theta -\kappa}+\va^{1 -\kappa} \sum\limits_{i=1}^k (V(y^i)-V(x^0))+\va e^{(-\min\{\frac p2,1\}-\kappa)\frac{|y^i-y^j|}{\va}}\big \}
$$
for any small $\kappa>0$.

Now, we verify that $A$ is a contraction mapping from $S_\va$ to
itself. For $\omega\in S_\va$, by Lemmas
\ref{lm3.1} and \ref{lm3.2}, we obtain
\begin{eqnarray*}
\|A(\omega)\|
&\leq&C(\|\bar{\ell}_\va\|+\|R_\va'(\omega)\|)\\
&\leq& C\big(\va^{1 +\theta }+\va \sum\limits_{i=1}^k (V(y^i)-V(x^0))+\va e^{-\min\{\frac p2,1\}\frac{|y^i-y^j|}{\va}}\big)\\
&\leq&  \va^{1+\theta -\kappa}+\va^{1 -\kappa} \sum\limits_{i=1}^k (V(y^i)-V(x^0))+\va e^{(-\min\{\frac p2,1\}-\kappa)\frac{|y^i-y^j|}{\va}}.
\end{eqnarray*}
Then,  $A$ maps $S_\va$ to $S_\va$.

On the other hand, for any $\omega_1,\omega_2 \in S_\va$,
\begin{eqnarray*}
\|A(\omega_1)-A(\omega_2)\|&=&\|L_\va^{-1}R_\va'(\omega_1)-L_\va^{-1}R_\va'(\omega_2)\|\\
&\leq& C\|R_\va'(\omega_1)-R_\va'(\omega_2)\|\\
&\leq& C\|R_\va''(\theta\omega_1+(1-\theta)\omega_2)\|\|\omega_1-\omega_2\|_{\va}\\
&\leq&\frac{1}{2}\|\omega_1-\omega_2\|_{\va}.
\end{eqnarray*}
So $A$ is a contraction map from $S_\va$ to $S_\va$. Consequently, applying the contraction mapping
theorem and implicit function
theorem, the conclusion is completed.

\end{proof}

Now, we are ready to prove our main theorem. Let
$\varphi_{\va,\bf {y}}$ be the map obtained in Proposition \ref{prop3.2}.

Define
$$
F({\bf {y} })=I(U_{*}+\varphi_{\va,{\bf {y} }}),\ \ \ \forall\, \varphi_\va\in S_\va.
$$
It is well known that if $\bf {y}$ is a critical point of $F(\bf {y})$, then $U_{*}+\varphi_{\va,{\bf {y} }}$ is a solution of our problem.

\begin{proof}[\textbf{Proof of Theorem \ref{th2}}] By Lemmas \ref{lm3.4} and \ref{lm3.1}, Propositions \ref{prop3.2} and \ref{energy-prop}, we find
\begin{eqnarray}\label{th1-1}
F({\bf {y}})&=& I_{\va}(U_{*})+\ell_{\va}(\varphi_{y'})+\frac{1}{2}\langle L_{\va}(\varphi_{{\bf {y}}}),\varphi_{{\bf {y}}}\rangle
+R_{\va}(\varphi_{{\bf {y}}})\nonumber\\
&=&I_\va(U_{\va,y})+O(\|\bar{\ell}_\va\|\|\varphi_{{\bf {y}}}\|+\|\varphi_{{\bf {y}}}\|^2)\nonumber\\
&=&A \va^2   - \frac12 \sik \big(V(x^0)-V(y^i)\big)  \va^2\int_{\r^2}U^2 dx - C_{1} \va^2\sum_{i\neq j}^ke^{\frac{-|y^i-y^j|}{\va}} \nonumber\\
& &+ O\bigl( \va^{2+\theta} +\va^2\sum\limits_{i\neq j}^ke^{\frac{-(1+\sigma)|y^i-y^j|}{\va}} + \va^{4}  \bigr) \nonumber\\
\end{eqnarray}

Consider the following maximizing problem
$$
\max\limits_{{\bf {y}}\in D_{\va,\delta}^k}F({\bf {y}} ).
$$

Since $F\in C^1,$ we can assume that $F$ is achieved by some ${{\bf {y}}_\va}$ in $\overline{D_{\va,\delta}^k}$. We will prove that ${\bf {y}}_\va$ is an interior point of $D_{\va,\delta}^k$.

Let $\bar{y}^i=x_0+M \va|\ln\va| e_i$, for some constant $M>0$, vectors $e_1,\cdots,e_k$ with $|e_i-e_j|=1$ for $i\neq j$.
Thus, for $M>0$ large, $\va$ small enough and $\bar y\in D_k^{\va,\delta}$, we have
\begin{eqnarray}\label{th1-2}
&&A\va^2-C_1\va^{2+\theta}|\ln\va|^\theta\nonumber\\
&\leq& A\va^2 -B_1\va^{2 +\theta}|\ln\va|^\theta -B_3\va^2 e^{-|\ln\va|}\nonumber\\
&\leq& F(\bar{{\bf {y}}}_\va)\leq F({\bf {y}}_\va)\nonumber\\
&\leq&A\va^2-C_2\va^2 \sum_{i=1}^k(V(x^0)-V(y^i)) -C_2\va^2 \sum_{i\neq j}^ke^{-\frac{|y^i-y^j|}{\va}}\nonumber\\
\end{eqnarray}
where $\bar{{\bf {y}}}_\va=(\bar{y}^1_\va,\cdots,\bar{y}^k_\va)$.

Employing \eqref{th1-2}, we deduce that for $\va$ sufficient small,
$$
A\va^2-C_1\va^{2+\theta}|\ln\va|^\theta\leq A\va^2-C_2\va^2\sum_{i=1}^k(V(x^0)-V(y^i)) -C_2\va^2\sum_{i\neq j}^ke^{-\frac{|y^i-y^j|}{\va}},
$$
i.e.
\begin{eqnarray*}
C_2\va^2 \sum_{i=1}^k(V(x^0)-V(y^i))+C_2\va^2\sum_{i\neq j}^ke^{-\frac{|y^i-y^j|}{\va}}
&\leq& C_1\va^{2 +\theta}|\ln\va|^\theta.
\end{eqnarray*}

That is,
$$
\sik \Bigl(V(y^i)-V(x^0)\Bigr)\leq C\va^{\theta}|\ln \va|^{\theta},
$$
$$
\sum\limits_{i\neq j}\frac{|y_\va^i-y_\va^j|}{\va}\geq \theta|\ln\va|\geq |\ln\va|^{\frac12}.
$$

This implies that ${\bf {y}}_\va$ is an interior point of $D_{\va,\delta}^k$ and hence is a critical point of $F({\bf {y}})$ for $\va$ sufficiently small.

Finally, by the standard argument and the strong maximum principle, we obtain that $u_{\va}=U_*+\varphi>0$.
\end{proof}

\section{Acknowledgements}{Long was supported by NSF of China (No. 11871253), NSF of Jiangxi Province (No. 20192ACB20012) and Jiangxi Two Thousand Talents Program. Yang was supported by NSF of China (No. 11671179 and 11771300).}

\end{document}